\documentclass{amsart}
\usepackage{amsfonts}
\usepackage{amsmath,amssymb}
\usepackage{amsthm}
\usepackage{amscd}
\usepackage{graphics}
\usepackage{graphicx}

\theoremstyle{remark}{
\newtheorem{Def}{{\rm Definition}}

}
\theoremstyle{plain}
{
\newtheorem{Cor}{Corollary}

\newtheorem{Thm}{Theorem}

}

\begin{document}
\title[Reeb graphs of real algebraic functions with nice trees]{Graphs with tree decompositions of small graphs and realizing them as the Reeb graphs of real algebraic functions}
\author{Naoki kitazawa}

\address{Osaka Central Advanced Mathematical Institute (OCAMI) \\
3-3-138 Sugimoto, Sumiyoshi-ku Osaka 558-8585
TEL: +81-6-6605-3103
}
\email{naokikitazawa.formath@gmail.com}
\urladdr{https://naokikitazawa.github.io/NaokiKitazawa.html}
\maketitle
\begin{abstract}
We have been interested in graphs and realizing them as Reeb graphs of explicit real algebraic functions.
The {\it Reeb graph} of a differentiable function is the quotient space of the manifold of the domain, regarded as the space consisting of all components of preimages of all single points. Reeb graphs have been fundamental and strong tools in geometry of manifolds since the birth of theory of Morse functions, in the former half of the 20th century. We can easily see that the Reeb graph of the natural height of the unit sphere whose dimension is at least $2$ is a graph with exactly one edge and two edges.

We are concerned with realizations of graphs decomposed into trees nicely, each vertex of which corresponds to a graph with exactly one edge and two vertices or a graph with exactly two edges homeomorphic to a circle.

\end{abstract}
\section{Introduction.}
\label{sec:1}
Since \cite{kitazawa2}, we have been interested in realizing graphs as the so-called Reeb graphs of real algebraic functions.
 
The {\it Reeb graph} of a differentiable function is a quotient space of the manifold of the domain and regarded as the space of all components of preimages of single points. Since the birth of the theory of so-called Morse functions. Reeb graphs have been fundamental and strong tools in geometry of manifolds. They have information of the manifolds roughly and nicely.

Realizing graphs as Reeb graphs of nice differentiable functions has been studied since \cite{sharko}. \cite{masumotosaeki} and \cite{michalak} follow this for example. The author has also contributed to this, considering not only graphs but also prescribed preimages of single points.

Our paper is also on construction of explicit real functions satisfying prescribed topological and combinatorial conditions.
Existence (and approximation) theory on real algebraic manifolds and maps is a branch of classical theory on real algebraic geometry, established mainly, by Nash and Tognoli \cite{nash, tognoli}. It is still developing. \cite{kollar} explains some related history. Systematic construction is another natural, difficult and important problem.

First, fundamental and important terminologies, notions, and notation on manifolds and maps between manifolds.

\subsection{Manifolds and maps in differentiable (smooth) situations and real algebraic ones.}
Hereafter, ${\mathbb{R}}^k$ denotes the $k$-dimensional Euclidean space, which is also the $k$-dimensional real vector space and the real affine space and a Riemannian manifold with the so-called standard Euclidean metric: for a point $x \in {\mathbb{R}}^k$, $x_j$ denotes the $j$-th component where $1 \leq j \leq k$ and for two points $x_1,x_2 \in {\mathbb{R}}^k$, $||x_1-x_2||:=\sqrt{{\Sigma}_{j=1}^k {(x_{1,j}-x_{2,j})}^2}$ denotes the distance of the two points under this metric. Let $||x||:=||x-0||$ with $0 \in {\mathbb{R}}^k$ denoting the origin. We define $D^k:=\{x \in {\mathbb{R}}^k \mid ||x|| \leq 1\}$, the $k$-dimensional unit disk, and $S^{k-1}:=\{x \in {\mathbb{R}}^k \mid ||x||=1\}$, the ($k-1$)-dimensional unit sphere.
We use ${\pi}_{m,n}:{\mathbb{R}}^m \rightarrow {\mathbb{R}}^n$ for the canonical projection, mapping $x=(x_1,x_2) \in {\mathbb{R}}^n \times {\mathbb{R}}^{m-n}={\mathbb{R}}^m$ to $x_1$ with $m>n \geq 1$. A real polynomial map $c:{\mathbb{R}}^m \rightarrow {\mathbb{R}}^n$ means a map each component $c_j:{\mathbb{R}}^m \rightarrow \mathbb{R}$ (the $j$-th component) of which is defined by a real polynomial with $m$ and $n$ being arbitrary positive integers. The canonical projection is a real polynomial map. 

For a differentiable manifold $X$, $T_x X$ denotes the tangent vector space of $X$ at $x \in X$. Given a differentiable map $c:X \rightarrow Y$ between the differentiable manifolds $X$ and $Y$, ${dc}_{x}:T_xX \rightarrow T_{c(x)} Y$ denotes the differential at $x$, which is also a linear map. A singular point $x \in X$ of $c$ is a point where the rank of the map ${dc}_x$ drops and $S(c)$ denotes the set of all singular points of $c$, the {\it singular set} of $c$. We consider smooth maps (maps of the class $C^{\infty}$) as differentiable maps unless otherwise stated. A {\it diffeomorphism} is a homeomorphism being smooth and having no singular point. Two smooth manifolds are {\it diffeomorphic} if a diffeomorphism between them exists.

A {\it Morse} function into a $1$-dimensional smooth manifold $C$ is a smooth function $c:X \rightarrow C$ which has no singular point on the boundary of the manifold and whose singular point $p$ is always of the form $c(x_1,\cdots x_m)={\Sigma}_{j=1}^{m-i(p)+1}{x_j}^2-{\Sigma}_{j=1}^{i(p)} {x_{m-i(p)+j}}^2$ for suitable local coordinates and a suitable integer $i(p)$: see \cite{milnor} and see also \cite{golubitskyguillemin}. A {\it Morse-Bott} function into $C$ is a smooth function at each singular point which is represented as the composition of a smooth function with no singular point with a Morse function: see \cite{bott}. 

We introduce real algebraic objects respecting existing sophisticated exposition presented in \cite{bochnakcosteroy, kollar} and our papers and preprints such as \cite{kitazawa2, kitazawa3, kitazawa4, kitazawa5, kitazawa6}.

A union of connected components of the zero set of a real polynomial map $c:{\mathbb{R}}^m \rightarrow {\mathbb{R}}^n$ is {\it non-singular} if the rank of $c$ does not drop at any point of $x \in c^{-1}(0) \subset {\mathbb{R}}^m$, where we respect the implicit function theorem.

The set represented as a union of connected components of the zero set of the map is a ({\it regular}) {real algebraic} manifold if it is non-empty and non-singular. The real affine space and the unit sphere are of simplest real algebraic manifolds.
Our real algebraic function (or generally, a map) means a function (resp. map) represented as the composition of the canonical embedding of a real algebraic manifold into the real affine space with either the identity map or the canonical projection to a lower dimensional real affine space.

The real affine space and the unit sphere with its canonical embedding to the real affine space and its composition with the canonical projection to a lower dimensional real affine space give simplest examples of these cases.
\subsection{Our main result.}
Before presenting our main result, Theorem \ref{thm:1}, we explain graphs shortly.
Our {\it graph} means a $j$-dimensional ($j=0,1$) connected and compact CW complex. This is also a finite CW complex. 
A {\it tree decomposition} of a graph which is 1-dimensional means a decomposition of graph into finitely many $1$-dimensional graphs satisfying certain connectivity distinct two graphs in which intersect in a single vertex-set if they intersect and which define a tree canonically. We consider the following small case: each of the finitely many graphs is a graph with exactly one edge and two vertices or a graph with exactly two edges and vertices which is homeomorphic to a circle. We call this a {\it simple cactus tree decomposition} or an {\it SCT decomposition}. A {\it tree} is a $1$-dimensional graph whose 1st Betti number is $0$ and characterized as a graph with an SCT decomposition consisting of only graphs with exactly one edge and two vertices (Corollary \ref{cor:1}).

The {\it Reeb graph} of a smooth function is a graph which is a quotient space of the manifold of the domain defined also as the space of all components of preimages of single points and whose vertex is a point corresponding to the components containing some singular points of the function.
Reeb graphs have been fundamental and strong tools in geometry of manifolds, using {\it Morse} functions (\cite{reeb}). 

We define related terminologies and notions rigorously, in the next section, again. We exhibit Theorem \ref{thm:1}. Here, we also refer to Theorem \ref{thm:3}, presented later. This is no problem.
\begin{Thm}
\label{thm:1}
Let $G_T$ be a tree.  
We have a real algebraic function represented as the composition of a real algebraic map as in Theorem \ref{thm:3} with the canonical projection ${\pi}_{2,1}$ satisfying the following conditions.
\begin{itemize}
\item The polynomials $f_j$ are of 1 or 2 degree and each $S_j$ is a circle of a fixed radius or a straight line in ${\mathbb{R}}^2$.
\item The function is a Morse-Bott function.
\item The Reeb graph is isomorphic to $G_T$.
\end{itemize}
\end{Thm}
Our related results on realizing graphs as the Reeb graphs of real algebraic functions are also on \cite{kitazawa3, kitazawa5, kitazawa6}. Compare this to them. 
\subsection{The organization of our paper.}
In the next section, we explain graphs, tree decompositions of graphs, and Reeb graphs, more precisely. We also prove Theorem \ref{thm:1}. We also present another result (Theorem \ref{thm:4}).

\section{On our main result.} 
We first explain terminologies, notions and notation we need, more precisely.
\subsection{Graphs, tree decompositions of graphs, and Reeb graphs.}
Our {\it graph} is a $0$-dimensional or $1$-dimensional CW complex which is finite and connected and whose underlying space is compact. An {\it edge} (a {\it vertex}) of a graph means a $1$-cell (resp. $0$-cell) of it. We also assume that our graph does not have an edge whose closure is homeomorphic to $S^1$. A {\it multigraph}, a graph which may have a pair of vertices connected by two distinct (closures of) edges, satisfies the definition of our graph.
The {\it edge} ({\it vertex}) {\it set} of a graph is the set of all edges (resp. vertices) of the graph.
 An {\it isomorphism} between two graphs is a piecewise smooth homeomorphism mapping the vertex set of a graph onto that of the other graph. Two graphs are {\it isomorphic} if an isomorphism between the graphs exists.

A {\it subgraph $G^{\prime}$} of a graph means a subcomplex of $G$ and this is also regarded as a graph.

A {\it tree} is a $1$-dimensional graph whose 1st Betti number is $0$.

\begin{Def}
\label{def:1}
A {\it tree decomposition} of a $1$-dimensional graph $G$ is a pair $(\{G_j\}_{j \in J},T_{\{G_j\}_{j \in J}})$ of a family $\{G_j\}_{j \in J}$ of 1-dimensional subgraphs of $G$ satisfying $G={\bigcup}_{j} G_j$ and a canonically defined graph $T_{\{G_j\}_{j \in J}}$ which is tree.
\begin{itemize}
\item By removing an edge of a subgraph $G_j$ of $G$, we have another graph. A graph we can obtain by this procedure is always connected. 
\item Each $G_j$ is maximal in the following sense.
If a graph ${G_{j}}^{\prime}$ is a subgraph of $G$ containing $G_j$ as a subgraph and if by removing an arbitrary edge of ${G_j}^{\prime}$ we always have a connected graph, then ${G_{j}}^{\prime}=G_j$.
\item Two distinct subgraphs $G_{i_1}$ and $G_{i_2}$ of $G$ are disjoint or intersect in a one-point set consisting of exactly one vertex of $G$. 
\item The graph $T_{\{G_j\}_{j \in J}}$ is defined as follows.
\begin{itemize}
\item The family $\{G_j\}_{j \in J}$ of subgraphs of $G$ is the vertex set of $T_{\{G_j\}_{j \in J}}$.
\item The distinct vertices $G_{i_1}$ and $G_{i_2}$, or equivalently, distinct subgraphs $G_{i_1}$ and $G_{i_2}$ of $G$ are connected by (the closure of) exactly one edge of $T_{\{G_j\}_{j \in J}}$ if the graphs $G_{i_1}$ and $G_{i_2}$ 
are not disjoint as subgraphs of $G$. The distinct vertices $G_{i_1}$ and $G_{i_2}$ are not connected by (the closure of) any edge of $T_{\{G_j\}_{j \in J}}$ if the graphs $G_{i_1}$ and $G_{i_2}$ are disjoint as subgraphs of $G$. 
\end{itemize}
\end{itemize}
\end{Def}
\begin{Thm}
\label{thm:2}
For any $1$-dimensional graph $G$, we have a tree decomposition uniquely. It is no problem that we use the notation $T_G:=T_{\{G_j\}_{j \in J}}$.
\end{Thm}

\begin{Def}
In Definition \ref{def:1}, if each subgraph $G_j$ is isomorphic to either of the following graph, then the tree decomposition of $G$ is called a {\it simple cactus tree decomposition} or an {\it SCT decomposition}.
\begin{itemize}
\item A graph with exactly one edge and two vertices.
\item A graph with exactly two edges and two vertices which is homeomorphic to $S^1$.
\end{itemize}
\end{Def}
\begin{Cor}
\label{cor:1}
A graph is a tree if and only if it has an SCT decomposition consisting of only graphs with exactly one edge and two vertices.
\end{Cor}
\begin{Def}
The {\it Reeb space} $W_c$ of a continuous map $c:X \rightarrow Y$ is defined in the following way. We can define an equivalence relation ${\sim}_c$ on $X$ by the following. Two points satisfy $x_1 {\sim}_c x_2$ if and only if they are in a same component of the preimage $c^{-1}(y)$ of some point $y$ and a continuous map $\bar{c}:W_C \rightarrow Y$ satisfying the relation $c=\bar{c} \circ q_c$ is uniquely defined.
For a smooth map $c:X \rightarrow Y$ on a closed manifold $X$ into a $1$-dimensional manifold $Y$ with no boundary, we have a graph according to \cite{saeki} where the vertex set of the Reeb space $W_c$ of $c$ is the set of all components $v$ with ${q_c}^{-1}(v)$ containing some points of $S(c)$. This is the {\it Reeb graph} of $c$.
\end{Def}
\subsection{Singularity theory of differentiable maps and differential topology and algebraic geometry.}
We explain fundamental singularity theory of differentiable maps and related differential topology, based on \cite{golubitskyguillemin}, for example.

For a real vector space $V$, $\dim V$ denotes its dimension. More generally, we use $\dim X$ for a topological manifold $X$. 

\begin{Def}
Two smooth submanifolds $X_1$ and $X_2$ with no boundaries in a smooth manifold $X$ with no boundary intersect in the {\it transversal way} in a subset $A \subset X$
if for each point $x \in X_1 \bigcap X_2 \subset A$, for the dimension $\dim (T_x X_1 \bigcap T_x X_2)$ of the vector space, we have $\dim (T_x X_1 \bigcap T_x X_2)=\dim T_x X_1+\dim T_x X_2-\dim T_x X$.
For example, in the case $\dim X_1+\dim X_2=\dim X$ with $X_1$ and $X_2$ being compact, the intersection $X_1 \bigcap X_2$ is a discrete subset of $X$. We consider this case mainly.
\end{Def}

The following is a kind of fundamental theorems in reconstructing real algebraic maps onto the closure of a given non-empty open set in a real affine space, first presented in \cite{kitazawa3}. This also extends an essential part in \cite{kitazawa2}, in which the case of manifolds $S_j$ being mutually disjoint is investigated. For \cite{kitazawa2}, see also \cite{kitazawa4}.
\begin{Thm}[\cite{kitazawa3, kitazawa5}]
\label{thm:3}
Let $D \subset {\mathbb{R}}^2$ be an open set which is non-empty satisfying the following.
\begin{itemize}
\item Let $\overline{D}$ be the closure of $D$ considered in ${\mathbb{R}}^2$. The set $\overline{D}-D$ is a subset of the union ${\bigcup}_{j=1}^{l} S_j$ of real algebraic manifolds $S_j$ {\rm (}$1 \leq j \leq l${\rm )} of dimension $1$ labeled by $l>0$ integers $j$.   
\item Two distinct real algebraic manifolds $S_{j_1} \subset {\mathbb{R}}^2$ and $S_{j_2} \subset  {\mathbb{R}}^2$ intersect in the {\it transversal way} in the subset $\overline{D} \subset {\mathbb{R}}^2$.
\item The intersection of three distinct real algebraic manifolds $S_{j_1}, S_{j_2}, S_{j_3} \subset  {\mathbb{R}}^2$ is always empty.
\item Let $S_j$ be a union of connected components of the zero set of a real polynomial function $f_j$. We have $S_j \bigcap D= D \bigcap \{x \mid f_j(x)=0\}$ and $D=\{x \mid f_j(x)>0\}$. The set $S_j \bigcap \overline{D}$ is non-empty.
\item For each integer $1 \leq j \leq l$, an integer $1 \leq m_{l}(j)=j^{\prime} \leq  l^{\prime}$ smaller than or equal to a (suitable) positive integer $L^{\prime}$ is assigned according to the rule that if $S_{j_1} \bigcap S_{j_2} \bigcap \overline{D}$ is non-empty, then $m_{l}(j_1)$ and $m_{l}(j_2)$ is not equal and that to any integer ${j_0}^{\prime}$ satisfying $1 \leq {j_0}^{\prime} \leq l^{\prime}$, we can assign at least one integer $j_0$ satisfying $m_{l_1}(j_0)={j_0}^{\prime}$. A non-negative integer $I_{j^{\prime}}$ is also defined for each integer $1 \leq j^{\prime} \leq l^{\prime}$. 
\end{itemize}
In this situation, the zero set $M:=\{(x,\{y_{I,j^{\prime}}\}_{j^{\prime}=1}^{l^{\prime}}) \in {\mathbb{R}}^2 \times {\prod}_{j^{\prime}=1}^{l^{\prime}} {\mathbb{R}}^{I_{j^{\prime}}+1}={\mathbb{R}}^{{\Sigma}_{j^{\prime}=1}^{l^{\prime}} (I_{j^{\prime}})+l^{\prime}+2} \mid {\prod}_{j_I \in \{j \mid m_{l_1}(j)=j^{\prime}\}} (f_{j_I}(x))-||y_{I,j^{\prime}}||^2=0, 1 \leq j^{\prime} \leq l^{\prime}\}$ is defined as a real algebraic manifold of dimension ${\Sigma}_{j^{\prime}=1}^{l^{\prime}} (I_{j^{\prime}})+l^{\prime}+2$. The restriction of the canonical projection ${\pi}_{{\Sigma}_{j^{\prime}=1}^{l^{\prime}} (I_{j^{\prime}})+l^{\prime}+2,2}$ to $M$ is defined as a real algebraic map onto $\overline{D} \subset {\mathbb{R}}^2$.\end{Thm}

Hereafter, $\mathbb{N} \subset \mathbb{R}$ denotes the set of all positive integers and ${\mathbb{N}}_{t} \subset \mathbb{N}$ denotes the set of all positive integers smaller than a given real number $t$.
Hereafter, it is fundamental and important to know that the function ${\pi}_{m+1,1} {\mid}_{S^m}$ is a Morse function with exactly two singular points. The restriction of this function to an $m$-dimensional (sub)manifold $X \subset S^m$ containing singular points of the original function in the interior is referred to as a {\it function for the natural height} of $X$.  

\begin{proof}[Reviewing our proofs of Theorem \ref{thm:3} in a self-contained way]
Originally, our proof is presented in a way respecting the implicit function theorem (\cite{kitazawa3}). We discuss our proof in another way, which is essentially same as that of our original discussion and is also discussed in \cite[Theorem 1]{kitazawa5}.

We assume fundamental arguments in real algebraic situations, discussed in \cite{kollar}, for example.

We consider local preimages of each point $p \in \overline{D}$ for the resulting map on $M$. Our main purpose is to see that $M$ is non-singular and a real algebraic manifold. \\
\ \\
Case 1.\ The case $p$ is in the interior $D$. \\
The preimage of a small open neighborhood diffeomorphic to $D^2$ is diffeomorphic to $D^2 \times {\prod}_{j^{\prime} \in {\mathbb{N}}_{l^{\prime}}} S^{j^{\prime}}$. The map gives a trivial smooth bundle. \\
\ \\
Case 2.\ The case $p$ is in $\overline{D}-D$ and contained in exactly one manifold $S_j$ of the family $\{S_j\}$. \\
The preimage of a small open neighborhood of $p$ diffeomorphic to $D^2$ is diffeomorphic to $D^1 \times D^{I_{j^{\prime}}+1} \times {\prod}_{j^{\prime \prime} \in {\mathbb{N}}_{l^{\prime}}-\{j^{\prime}\}} S^{I_{j^{\prime \prime}}}$. The map gives the product map of a Morse function for the natural height of the disk and the identity map on the product of a copy of $D^1$ and finitely many copies of the unit spheres.\\
\ \\
Case 3.\ The case $p$ is in $\overline{D}-D$ and contained in exactly two manifolds $S_{j_1}$ and $S_{j_2}$ of the family $\{S_j\}$. \\
The preimage of a small open neighborhood of $p$ diffeomorphic to $D^{I_{{j_1}^{\prime}}+1} \times D^{I_{{j_2}^{\prime}}+1} \times {\prod}_{j^{\prime \prime} \in {\mathbb{N}}_{l^{\prime}}-\{{j_1}^{\prime},{j_2}^{\prime}\}} S^{I_{j^{\prime \prime}}}$. The map gives the product map of two Morse functions for the natural height of two disks and the identity map on the product of finitely many copies of the unit spheres.\\

We can check that $M$ is a real algebraic manifold.

This completes our proof.
\end{proof}
\subsection{Proving Theorem \ref{thm:1} and another new result, Theorem \ref{thm:4}.}
The {\it straight line} in ${\mathbb{R}}^2$ means the zero set of a real polynomial of degree $1$. The {\it circle} ({\it centered at $p$} and {\it of a radius $r>0$}) means the real algebraic manifold $\{x \in {\mathbb{R}}^2 \mid ||x-p||=r\}$, diffeomorphic to $S^1$. 
We also assume fundamental knowledge on geometry of elementary figures in the Euclidean plane ${\mathbb{R}}^2$.  
\begin{proof}[A proof of Theorem \ref{thm:1}.]
A tree $G_T$ admits a piecewise smooth function $g_{G_T}:G_T \rightarrow \mathbb{R}$ with the following properties. This is a kind of well-known theorems on graphs.
\begin{itemize}
\item The restriction of $g_{G_T}$ to each edge of $G_T$ is a smooth embedding.
\item There exists a piecewise smooth embedding $\tilde{g_{G_T}}:G_T \rightarrow {\mathbb{R}}^2$ with $g_{G_T}={\pi}_{2,1} \circ \tilde{g_{G_T}}$.
\item Let $\{i_{g_{G_T},j}\}_{j=1}^{I_{g_{G_T}}}$ denote the image of the restriction of the function to the vertex set of $G_T$ where $I_{g_{G_T}}$ denotes a positive integer. We choose $i_{g_{G_T},a}<p_a<i_{g_{G_T},a+1}$. The size of the preimage ${g_{G_T}}^{-1}(p_{1})$ is $1$. The size of the preimage ${g_{G_T}}^{-1}(p_{j+1})$ is equal to that of the preimage ${g_{G_T}}^{-1}(p_{j})$ or the sum of the size of the preimage ${g_{G_T}}^{-1}(p_{j})$ and a suitable positive integer. 

Note that the size of the preimage ${g_{G_T}}^{-1}(p_{a})$ is the number of the edges of $G_T$ mapped onto $\{p \mid i_{g_{G_T},a} < p < i_{g_{G_T},a+1}\}$ and let $\{e_{g_{G_T},a,j}\}_{j=1}^{I_{g_{G_T},a}}$ denote the set of all such edges of $G_T$.
\end{itemize} 
FIGURE \ref{fig:1} visually presents the function $g_{G_T}:G_T \rightarrow \mathbb{R}$ and the embedding $\tilde{g_{G_T}}:G_T \rightarrow {\mathbb{R}}^2$ without the related notation.
\begin{figure}
		\includegraphics[width=70mm,height=25mm]{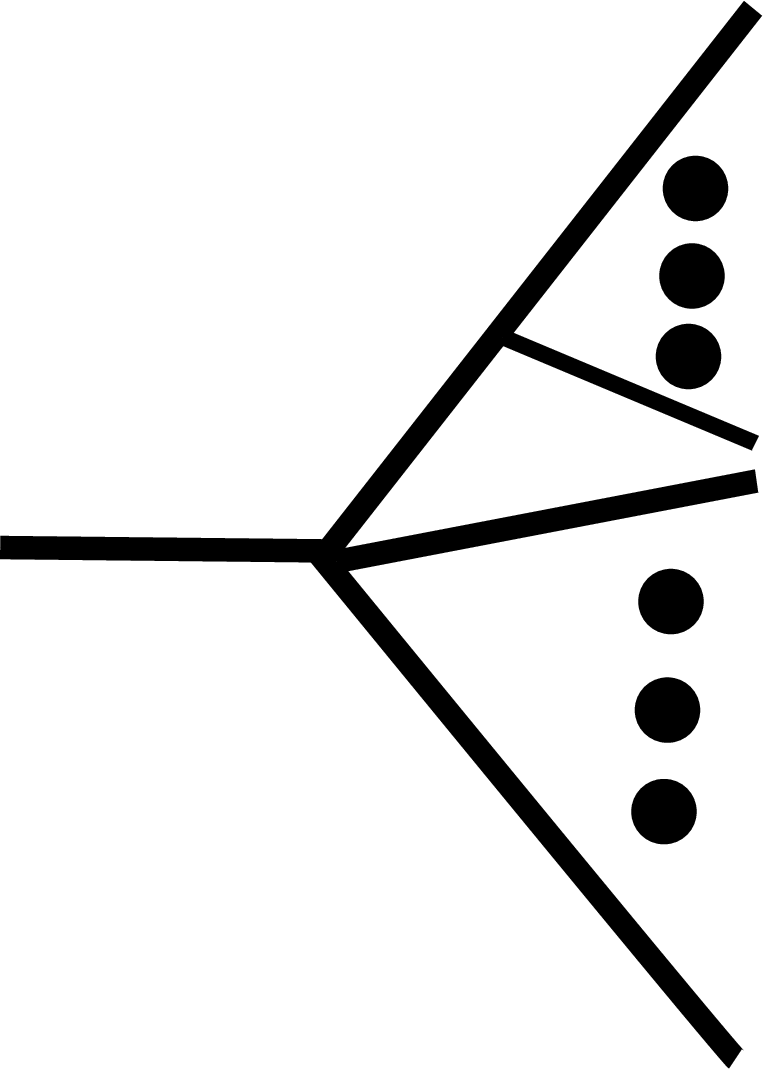}
		\caption{A piecewise smooth function $g_{G_T}:G_T \rightarrow \mathbb{R}$ (a piecewise smooth embedding $\tilde{g_{G_T}}:G_T \rightarrow {\mathbb{R}}^2$).}
                      \label{fig:1}

		\end{figure}

By choosing the values $i_{g_{G_T},j}$ and real numbers $Y_1<Y_2$ suitably, we can choose the following straight lines and circles of fixed radii surrounding the graph $\tilde{g_{G_T}}(G_T)$ inside the open region $D \subset {\mathbb{R}}^2$ formed by them and additional data for Theorem \ref{thm:3}. We abuse the notation from Theorem \ref{thm:3} and scenes around this.
\begin{itemize}
\item We define $S_1:=\{(i_{g_{G_T},1}-{\epsilon}_1,t)  \mid t \in \mathbb{R}\}$ for a suitably chosen real number ${\epsilon}_1>0$.
\item We define $S_2:=\{(i_{g_{G_T},I_{g_{G_T}}}+{\epsilon}_{I_{g_{G_T}}},t)  \mid t \in \mathbb{R}\}$ for a suitably chosen real number ${\epsilon}_{I_{g_{G_T}}}>0$.
\item We define $S_3:=\{(t,Y_1) \mid t \in \mathbb{R}\}$.
\item We define $S_4:=\{(t,Y_2) \mid t \in \mathbb{R}\}$.
\item We define $S_{4+j}:=\{||x-p_{S_2,j}||=r_j \mid x \in {\mathbb{R}}^2\}$ is a circle of a fixed radius $r_j>0$ which is centered at a point $p_{S_2,j} \in S_2$ and which does not intersect $S_1$, $S_3
$ or $S_4$, for each integer $1 \leq j \leq I_{g_{G_T},I_{g_{G_T}}-1} -1$: moreover, distinct circles $S_{4+j}$ are disjoint and the disks obtained as the disks whose boundaries are the circles are mutually disjoint.
\item Remaining $l-(I_{g_{G_T},I_{g_{G_T}}-1}+3)$ circles $S_{I_{g_{G_T},I_{g_{G_T}}-1}+3+j}$ ($1 \leq j \leq l_1-(I_{g_{G_T},I_{g_{G_T}}-1}+3)$) are added satisfying the following rules.
\begin{itemize}
\item These remaining circles are circles of fixed radii and mutually disjoint and the disks obtained as the disks whose boundaries are the circles are also mutually disjoint. The intersection of each circle $S_{I_{g_{G_T},I_{g_{G_T}}-1}+3+j}$ ($1 \leq j \leq l-(I_{g_{G_T},I_{g_{G_T}}-1}+3)$) and ${\bigcup}_{j=1}^{I_{g_{G_T},I_{g_{G_T}}-1}+3} S_j$ is a discrete set consisting of two points.
\item For each edge $e_{g_{G_T},a,j}$ and the image $\tilde{g_{G_T}}(e_{g_{G_T},a,j})$, we can choose a suitable small number ${\epsilon}_a>0$ and have  the connected component  $D_{e_{g_{G_T},a,j}}$ 
of the intersection of the preimage of the interval $\{p \mid i_{g_{G_T},a}+{\epsilon}_a<p<i_{g_{G_T},a+1}\}$ by the projection ${\pi}_{2,1}$ and $D$ containing $e_{g_{G_T},a,j}$ uniquely. Furthermore, the closure $\overline{D_{e_{g_{G_T},a,j}}}$ of the connected component considered in ${\mathbb{R}}^2$ and at most one circle $S_{I_{g_{G_T},I_{g_{G_T}}-1}+3+j}$ ($1 \leq j \leq l-(I_{g_{G_T},I_{g_{G_T}}-1}+3)$) intersect and for such a circle $S_{I_{g_{G_T},I_{g_{G_T}}-1}+3+j}$, the set $\overline{D_{e_{g_{G_T},a,j}}} \bigcap S_{I_{g_{G_T},I_{g_{G_T}}-1}+3+j}$ is diffeomorphic to $D^1$ and mapped onto $\{p \mid i_{g_{G_T},a} \leq p \leq i_{g_{G_T},a+1}\}$ by a diffeomorphism. Furthermore, the following are satisfied.
\begin{itemize}
\item The set $\overline{D_{e_{g_{G_T},a,j}}} \bigcap ({\bigcup}_{j=1}^{l_1} S_j) \bigcap {{\pi}_{2,1}}^{-1}(i_{g_{G_T},a}+{\epsilon}_a)$ consists of two points. 
\item The circle $S_{I_{g_{G_T},I_{g_{G_T}}-1}+3+j}$ ($1 \leq j \leq l-(I_{g_{G_T},I_{g_{G_T}}-1}+3)$) can not be found as the argument before if two distinct circles $S_{4+j}$ ($1 \leq j \leq I_{g_{G_T},I_{g_{G_T}}-1} -1$) of a same radius $r_j$ containing points of $\overline{D_{e_{g_{G_T},a,j}}} \bigcap ({\bigcup}_{j=1}^{l} S_j) \bigcap {{\pi}_{2,1}}^{-1}(i_{g_{G_T},a}+{\epsilon}_a)$ exist. 
\item Each circle $S_{I_{g_{G_T},I_{g_{G_T}}-1}+3+j}$ ($1 \leq j \leq l-(I_{g_{G_T},I_{g_{G_T}}-1}+3)$) is chosen as before for at most one edge $e_{g_{G_T},a,j}$ of $G_T$.
\item For the edge $e_{g_{G,T},1,1}$, we can choose a circle from the family of circles $S_{I_{g_{G_T},I_{g_{G_T}}-1}+3+j}$ ($1 \leq j \leq l-(I_{g_{G_T},I_{g_{G_T}}-1}+3)$) as before.
\item For a vertex of degree $2$, for at least one edge containing the vertex, we can choose a circle from the family of circles $S_{I_{g_{G_T},I_{g_{G_T}}-1}+3+j}$ ($1 \leq j \leq l-(I_{g_{G_T},I_{g_{G_T}}-1}+3)$) as before. This is for the Reeb graph of a desired function to have its desired vertex set and the function to have the singular set of our desired type.
\end{itemize}  
\end{itemize}
\item We put $m_{l}(j_1)=1$ for $1 \leq j_1 \leq 2$, $m_{l}(j_2)=2$ for $3 \leq j_2 \leq 4$, $m_{l}(j_3)=3$ for $5 \leq j_3 \leq I_{g_{G_T},I_{g_{G_T}}-1}+3$, and $m_{l}(j_4)=4$ for the remaining integers $j_4 \leq l_1$. We also choose $I_1, I_2, I_3, I_4>0$.
\end{itemize}
For this, see also FIGURE \ref{fig:2}.
\begin{figure}
		\includegraphics[width=70mm,height=25mm]{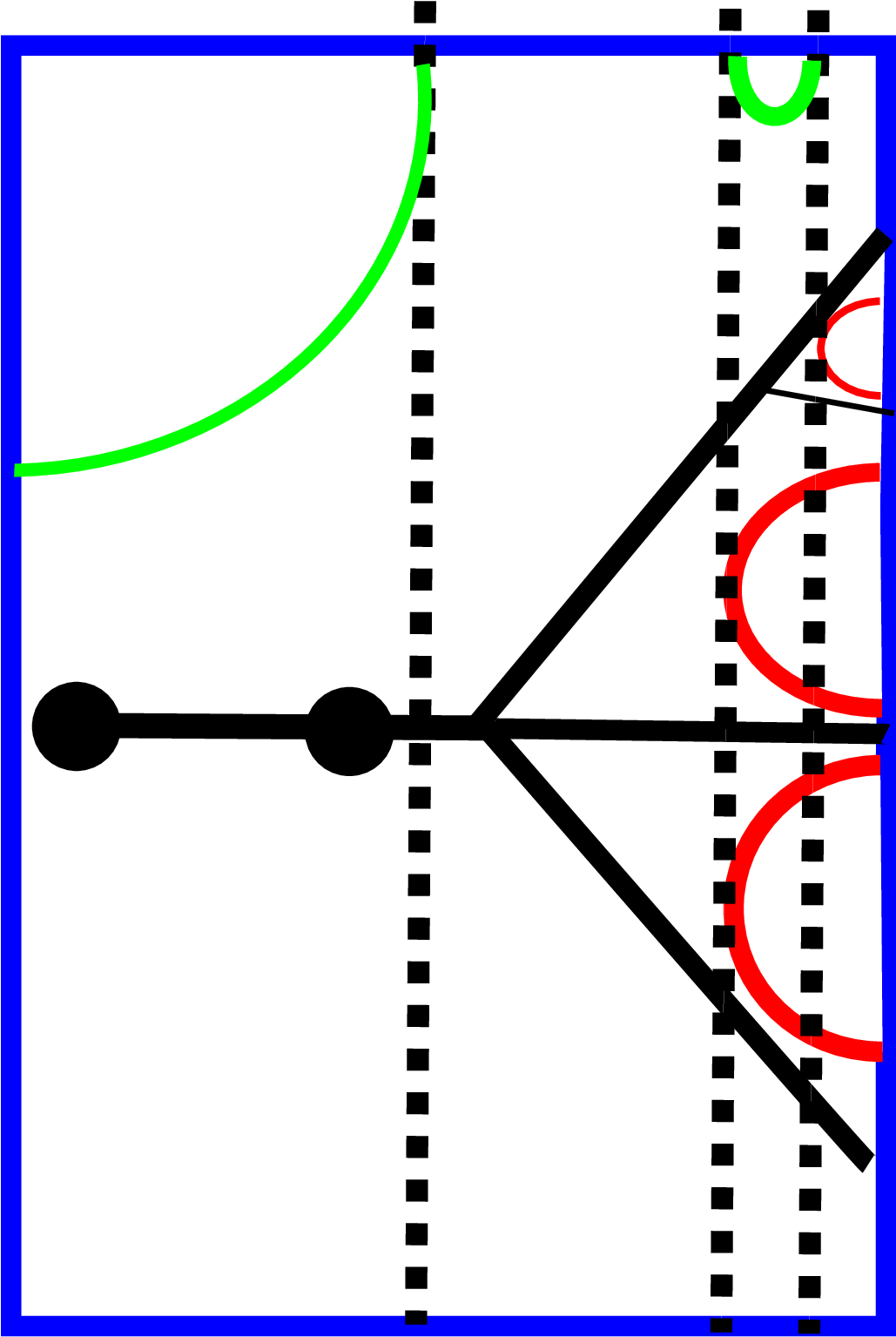}
		\caption{We present a case with $l=9$ , $l^{\prime}=3$, and $I_{g_{G_T}}=5$.
The lines $S_j$ are colored in blue ($1 \leq j \leq 4$). We have $I_{g_{G_T},I_{g_{G_T}}-1}=I_{g_{G_T},4}=4$.
The circles $S_{4+j}$ are colored in red ($1 \leq j \leq I_{g_{G_T},I_{g_{G_T}}-1}-1=4-1=3$). The circles $S_{I_{g_{G_T},I_{g_{G_T}}-1}+3+j}$ ($1 \leq j \leq l-(I_{g_{G_T},I_{g_{G_T}}-1}+3)=9-(4+3)=9-7=2$) are colored in green.}
\label{fig:2}
		\end{figure}
By Theorem \ref{thm:3} with a fundamental argument on the singularity theory, explained in \cite{golubitskyguillemin} for example, and a fundamental argument on real algebraic geometry, explained in \cite{bochnakcosteroy, kollar} (, especially in \cite[Discussion 14]{kollar}), we have a desired real algebraic Morse-Bott function. To know more precise exposition, consult also the preprint \cite{kitazawa3, kitazawa5, kitazawa6}.

This completes the proof.
\end{proof}
\begin{Thm}
\label{thm:4}
Let $G$ be either of the following graphs.
\begin{itemize}
\item A graph $G_0$ having an SCT decomposition $T_{G_0}$ satisfying the following.
\begin{itemize}
\item Hereafter, we abuse the notation in the proof of Theorem \ref{thm:1}. There exists a function $g_{G_T}:G_T \rightarrow \mathbb{R}$ as in the proof.
\item The vertex set of $T_{G_0}$ corresponds to the edge set of $G_T$ and we respect the correspondence. Two distinct vertices of $T_{G_0}$ are connected by the closure of an edge of the graph $T_{G_0}$ if and only if the closures of the edges of $G_T$ intersect in a one-point set.
\item For an edge of $G_T$ we can choose the circle $S_{i_{g_{G_T},I_{g_{G_T}}-1}+3+j}$ {\rm (}$1 \leq j \leq l_1-(i_{g_{G_T},I_{g_{G_T}}-1}+3)${\rm )} in the proof, the corresponding vertex of $T_{G_0}$ is a graph with exactly one edge.
\item  For an edge of $G_T$ we cannot choose the circle $S_{i_{g_{G_T},I_{g_{G_T}}-1}+3+j}$ {\rm (}$1 \leq j \leq l_1-(i_{g_{G_T},I_{g_{G_T}}-1}+3)${\rm  )}in the proof, the corresponding vertex of $T_{G_0}$ is a graph homeomorphic to $S^1$.
\end{itemize}  
\item A graph $G$ obtained from the graph $G_0$, defined above, in the following way.
\begin{itemize}
\item Choose mutually disjoint sets of two or three edges of the graph $G_T$ for which we can choose the circles $S_{i_{g_{G_T},I_{g_{G_T}}-1}+3+j}$ {\rm (}$1 \leq j \leq l-(i_{g_{G_T},I_{g_{G_T}}-1}+3)${\rm )} as in the proof of Theorem \ref{thm:1}. Choose each set as a set with either of the following forms.
\begin{itemize}
\item $\{e_{g_{G_T},1,1},e_{g_{G_T},2,j}\}$.
\item $\{e_{g_{G_T},I_{g_{G_T}}-2,j_1},e_{g_{G_T},I_{g_{G_T}}-1,j_2}\}$.
\item $\{e_{g_{G_T},a,j_1},e_{g_{G_T},a+1,j_2},e_{g_{G_T},a+2,j_3}\}$.
\end{itemize}
\item We do the following for each of the mutually disjoint sets before and change the tree $T_{G_0}$ into another tree $T_G$. The new graph $G$ is a graph which has an SCT decomposition $T_G$.
\begin{itemize}
\item For the set of the form $\{e_{g_{G_T},1,1},e_{g_{G_T},2,j}\}$, we change the vertex of $T_{G_0}$ corresponding to $e_{g_{G_T},1,1}$ into a graph homeomorphic to $S^1$.
\item For the set of the form $\{e_{g_{G_T},I_{g_{G_T}}-2,j_1},e_{g_{G_T},I_{g_{G_T}}-1,j_2}\}$, we change the vertex of $T_{G_0}$ corresponding to $e_{g_{G_T},I_{g_{G_T}}-1,j_2}$ into a graph homeomorphic to $S^1$.
\item For the set of the form $\{e_{g_{G_T},a,j_1},e_{g_{G_T},a+1,j_2},e_{g_{G_T},a+2,j_3}\}$, we change the vertex of $T_{G_0}$ corresponding to $e_{g_{G_T},a+1,j_2}$ into a graph homeomorphic to $S^1$.
\end{itemize}
\end{itemize}
\end{itemize}
We have a real algebraic function represented as the composition of a real algebraic map as in Theorem \ref{thm:3} with the canonical projection ${\pi}_{2,1}$ satisfying the following conditions.
\begin{itemize}
\item The polynomials $f_j$ are of 1 or 2 degree and each $S_j$ is a circle of a fixed radius or a straight line in ${\mathbb{R}}^2$.
\item The function is a Morse-Bott function.
\item The Reeb graph is isomorphic to $G$.
\end{itemize}
\end{Thm}
\begin{proof}
In the proof of Theorem \ref{thm:1}, we put $I_4=0$ instead and consider the function obtained in the same way. 
We can check the first case by investigating the preimages of each point for the obtained function.

We prove the second case. For the edge $e_{g_{G_T},1,1} \in \{e_{g_{G_T},1,1},e_{g_{G_T},2,j}\}$, $e_{g_{G_T},I_{g_{G_T}}-1,j_2} \in \{e_{g_{G_T},I_{g_{G_T}}-2,j_1},e_{g_{G_T},I_{g_{G_T}}-1,j_2}\}$, and $e_{g_{G_T},a+1,j_2} \in \{e_{g_{G_T},a,j_1},e_{g_{G_T},a+1,j_2},e_{g_{G_T},a+2,j_3}\}$ of $G_T$, mentioned in the end of the statement, we choose the corresponding circle $S_{i_{g_{G_T},I_{g_{G_T}}-1}+3+j}$ {\rm (}$1 \leq j \leq l-(i_{g_{G_T},I_{g_{G_T}}-1}+3)${\rm )} chosen for the edge of $G_T$, in the proof of Theorem \ref{thm:1}. We omit these circles in applying Theorem \ref{thm:3} in Theorem \ref{thm:1}. After this argument, we have a desired function.

This completes the proof. 
\end{proof}
\section{Conflict of interest and data.}
The author is a researcher at Osaka Central
Advanced Mathematical Institute (OCAMI researcher). This is also supported by MEXT Promotion of Distinctive Joint Research Center Program JPMXP0723833165. The author is not employed in the institution. However, the author thanks this. \\

No other data are associated to the present paper.

\end{document}